\documentclass[a4paper, 12pt]{amsart}
\usepackage{amsmath, amssymb, amsthm, microtype}
\usepackage[bottom=1.5in, left=1.25in, right=1.25in, top=1.5in]{geometry}

\newtheorem{theorem}{Theorem}[section]

\newtheorem{lemma}[theorem]{Lemma}

\newtheorem{proposition}[theorem]{Proposition}

\theoremstyle{remark}
\newtheorem{remark}[theorem]{Remark}

\theoremstyle{definition}

\newcommand{\cC}{\mathcal{C}}

\newcommand{\cH}{\mathcal{H}}
\newcommand{\cP}{\mathcal{P}}
\newcommand{\cR}{\mathcal{R}}

\newcommand{\bN}{\mathbb{N}}
\newcommand{\bF}{\mathbb{F}}
\newcommand{\bP}{\mathbb{P}}
\newcommand{\bQ}{\mathbb{Q}}
\newcommand{\bZ}{\mathbb{Z}}
\newcommand{\lcm}{\operatorname{lcm}}
\newcommand{\Frob}{\operatorname{Frob}}
\newcommand{\Gal}{\operatorname{Gal}}
\newcommand{\res}{\operatorname{res}}
\newcommand{\ord}{\operatorname{ord}}

\title{Bounded gaps between prime polynomials with a given primitive root}

\begin{document}

\author{Lee Troupe}
\address{Department of Mathematics, Boyd Graduate Studies Research Center, University of Georgia, Athens, GA 30602, USA}
\email{ltroupe@math.uga.edu}
\let\thefootnote\relax\footnote{The author was supported by NSF RTG Grant DMS-1344994.}



%

\begin{abstract}
A famous conjecture of Artin states that there are infinitely many prime numbers for which a fixed integer $g$ is a primitive root, provided $g \neq -1$ and $g$ is not a perfect square. Thanks to work of Hooley, we know that this conjecture is true, conditional on the truth of the Generalized Riemann Hypothesis. Using a combination of Hooley's analysis and the techniques of Maynard-Tao used to prove the existence of bounded gaps between primes, Pollack has shown that (conditional on GRH) there are bounded gaps between primes with a prescribed primitive root. In the present article, we provide an unconditional proof of the analogue of Pollack's work in the function field case; namely, that given a monic polynomial $g(t)$ which is not an $v$th power for any prime $v$ dividing $q-1$, there are bounded gaps between monic irreducible polynomials $P(t)$ in $\mathbb{F}_q[t]$ for which $g(t)$ is a primitive root (which is to say that $g(t)$ generates the group of units modulo $P(t)$).  In particular, we obtain bounded gaps between primitive polynomials, corresponding to the choice $g(t) = t$.
\end{abstract}

%

\maketitle

\section{Introduction}

Among the most prominent conjectures in number theory is the prime $k$-tuples conjecture of Hardy and Littlewood, the qualitative version of which states that for any admissible tuple of integers $\cH = \{h_1, \ldots, h_k\}$, there are infinitely many natural numbers $n$ such that the shifted tuple $n + \cH = \{n + h_1, \ldots, n + h_k\}$ consists entirely of primes. To this day, we do not know of a single admissible tuple for which the above statement is true.

We can instead ask for something weaker: Can we show that infinitely many shifts of admissible $k$-tuples contain just two or more primes? In 2013, Yitang Zhang stunned the mathematical world by demonstrating that, for every sufficiently long tuple $\cH$, there are infinitely many natural numbers $n$ for which $n + \cH$ contains at least two primes, thereby establishing the existence of infinitely many bounded gaps between consecutive primes \cite{zhang14}. Zhang's breakthrough was soon followed by work of Maynard \cite{may15} and Tao, who independently established that infinitely many shifts of admissible $k$-tuples contain $m$ primes, for any $m \geq 2$, provided $k$ is large enough with respect to $m$. As a consequence, we have not only bounded gaps between primes, but also that $\liminf_{n \to \infty} p_{n+m} - p_{n} < \infty$ (here $p_n$ denotes the $n$th prime number).

The Maynard-Tao machinery can be utilized to probe questions concerning bounded gaps between primes in other contexts. Let $q$ be a power of a prime and consider the ring $\bF_q[t]$. We say that an element $p$ of $\bF_q[t]$ is \emph{prime} if $p$ is monic and irreducible. The following theorem, a bounded gaps result for $\bF_q[t]$, is due to Castillo, Hall, Lemke Oliver, Pollack, and Thompson \cite{chlopt}.

\begin{theorem}\label{chlopt}
Let $m \geq 2$. There exists an integer $k_0$ depending on $m$ but independent of $q$ such that for any admissible $k$-tuple $\{h_1, \ldots, h_k\} \subset \bF_q[t]$ with $k \geq k_0$, there are infinitely many $f \in \bF_q[t]$ such that at least $m$ of $f + h_1, \ldots, f + h_k$ are prime.
\end{theorem}

In particular, if $\{h_1, \ldots, h_k\} \subset \bF_q[t]$ is a long enough admissible tuple, the difference in norm between primes in $\bF_q[t]$ is at most $\max_{1 \leq i \neq j \leq k} |h_i - h_j|$, infinitely often. (Here $|f| = q^{\deg f}$ for $f \in \bF_q[t]$.)

Artin's famous primitive root conjecture states that for any integer $g \neq -1$ and not a square, there are infinitely many primes for which $g$ is a primitive root; that is, there are infinitely many primes $p$ for which $g$ generates $(\bZ/p\bZ)^\ast$. Work of Hooley \cite{hoo67} establishes the truth of Artin's conjecture, assuming GRH; the following result due to Pollack \cite{MR3272281} is a bounded gaps result in this setting.

\begin{theorem}[conditional on GRH]\label{pollackprimroots}
Fix an integer $g \neq -1$ and not a square. Let $q_1 < q_2 < \ldots$ denote the sequence of primes for which $g$ is a primitive root. Then for each $m$,
\[
\liminf_{n \to \infty} (q_{n + m - 1} - q_n) \leq C_m,
\]
where $C_m$ is finite and depends on $m$ but not on $g$.
\end{theorem}

Artin's conjecture can be formulated in the setting of polynomials over a finite field with $q$ elements, where $q$ is a prime power. Let $g \in \bF_q[t]$ be monic and not an $v$th power, for any $v$ dividing $q-1$; this is analogous to the requirement that $g$ not be a square in the integer case. We say that $g$ is a primitive root for a prime polynomial $p \in \bF_q[t]$ if $g$ generates the group $(\bF_q[t]/p\bF_q[t])^{\ast}$. In Bilharz's 1937 Ph.D. thesis \cite{bil37}, he confirms Artin's conjecture that there are infinitely many such $p$ for a given $g$ satisfying the above requirements, conditional on the Riemann hypothesis for global function fields, a result proved by Weil in 1948.

Motivated by the results catalogued above, we presently establish an unconditional result which can be viewed as a synthesis of Theorems \ref{chlopt} and \ref{pollackprimroots}.

\begin{theorem}\label{main}
Let $g$ be a monic polynomial in $\bF_q[t]$ such that $g$ is not a $v$th power for any prime $v$ dividing $q - 1$, and let $\bP_g$ denote the set of prime polynomials in $\bF_q[t]$ for which $g$ is a primitive root. For any $m \geq 2$, there exists an admissible $k$-tuple $\{h_1, \ldots, h_k\}$ such that there are infinitely many $f \in \bF_q[t]$ with at least $m$ of $f + h_1, \ldots, f + h_k$ belonging to $\bP_g$.
\end{theorem}

\begin{remark}
A prime polynomial $a \in \bF_q[t]$ is called \emph{primitive} if $t$ is a primitive root for $a$; see \cite{ln97} for an overview of primitive polynomials. Taking $g = t$, we obtain as an immediate corollary the existence of bounded gaps between primitive polynomials.
\end{remark}

\noindent\textbf{Notation.} In what follows, $q$ is an arbitrary but fixed prime power and $\bF_q$ is the finite field with $q$ elements. The Greek letter $\Phi$ will denote the Euler phi function for $\bF_q[t]$; that is, $\Phi(f) = \#(\bF_q[t]/f\bF_q[t])^\ast$. The symbols $\ll, \gg,$ and the $O$ and $o$-notations have their usual meanings; constants implied by this notation may implicitly depend on $q$. Other notation will be defined as necessary.

\section{The necessary tools}

For a monic polynomial $a$ and a prime polynomial $P$ not dividing $a$ in $\bF_q[t]$, define the $d$-th power residue symbol $(a/P)_d$ to be the unique element of $\bF_q^{\ast}$ such that
\[
a^{\frac{|P| - 1}{d}} \equiv \left(\frac{a}{P}\right)_d \pmod P.
\]
Let $b \in \bF_q[t]$ be monic, and write $b = P_1^{e_1} \cdots P_s^{e_s}$. Define
\[
\left(\frac{a}{b}\right)_d = \prod_{j = 1}^s \left(\frac{a}{P_j}\right)_d^{e_j}.
\]
We will make use of a number of properties of the $d$-th power residue symbol. The following is taken from Propositions 3.1 and 3.4 of \cite{ros02}.

\begin{proposition}\label{properties}
The $d$-th power residue symbol has the following properties.
\begin{itemize}
\item[(a)] $\left(\frac{a_1}{b}\right)_d = \left(\frac{a_2}{b}\right)_d$ if $a_1 \equiv a_2 \pmod b$. 
\bigskip
\item[(b)]Let $\zeta \in \bF_q^{\ast}$ be an element of order dividing $d$. Then, for any prime $P \in \bF_q[t]$ with $P \nmid a$, there exists $a \in \bF_q[t]$ such that $\left(\frac{a}{P}\right)_d = \zeta$.
\end{itemize}
\end{proposition}

We now state a special case of the general reciprocity law for $d$-th power residue symbols in $\bF_q[t]$, Theorem 3.5 in \cite{ros02}:

\begin{theorem}\label{reciprocitylaw}
Let $a, b \in \bF_q[t]$ be monic, nonzero and relatively prime. Then
\[
\Big(\frac{a}{b}\Big)_d = \Big(\frac{b}{a}\Big)_d (-1)^{\frac{q-1}{d} \deg(a)\deg(b)}.
\]
\end{theorem}

Another essential tool in our analysis is the Chebotarev density theorem. The following is a restatement of Proposition 6.4.8 in \cite{fj08}.

\begin{theorem}\label{chebotarev}
Write $K = \bF_q(t)$ and let $L$ be a finite Galois extension of $K$, and let $\cC$ be a conjugacy class of $\Gal(L/K)$. Let $\bF_{q^n}$ be the constant field of $L/K$. For each $\tau \in \cC$, suppose $\res_{\bF_{q^n}} \tau = \res_{\bF_{q^n}} \Frob_q^k$, where $k \in \bN$. The number of unramified primes $P$ of degree $k$ whose Artin symbol $\left(\frac{L/K}{P}\right)$ is $\cC$ is given by
\[
\frac{\#\cC}{m}\frac{q^k}{k} + O\left(\frac{\#\cC}{m}\frac{q^{k/2}}{k}(m + g_L)\right), 
\]
where $m = [L : K\bF_{q^n}]$, $g_L$ is the genus of $L/K$, and the constant implied by the big-$O$ is absolute.
\end{theorem}

In our application, the extension $L/K$ will be the compositum of a Kummer extension and a cyclotomic extension of $K = \bF_q(t)$. The next three results will help us estimate $g_L$.

We say that an element $a \in K^\ast$ is \emph{geometric} at a prime $r \neq q$ if $K(\sqrt[r]{a})$ is a geometric field extension of $K$ (that is, the constant field of $K(\sqrt[r]{a})$ is the same as the constant field of $K$). Proposition 10.4 in \cite{ros02} concerns the genus of such extensions; we state it below.

\begin{proposition}\label{kummergenus}
Suppose $r \neq \operatorname{char} \bF_q$ is a prime and $K' = K(\sqrt[r]{a})$, $a \in K$ nonzero. Assume that $a$ is geometric at $r$ and that $a$ is not an $r$th power in $K^\ast$. With $g_{K'}$ denoting the genus of $K'/K$,
\[
2g_{K'} - 2 = -2r + R_a(r - 1),
\]
where $R_a$ is the sum of the degrees of the finitely many primes $P \in K$ where the order of $P$ in $a$ is not divisible by $r$.
\end{proposition}

Fix an algebraic closure $\overline{K}$ of $K$. One can define an analogue of exponentiation in $\bF_q[t]$; that is, for $M \in \bF_q[t]$ and $u \in \overline{K}$, the symbol $u^M$ is again an element of $\overline{K}$. In particular, we have an analogue of cyclotomic field extensions. Define $\Lambda_M = \{u \in \overline{K} \mid u^M = 0\}$; then $K(\Lambda_M)/K$ is a \emph{cyclotomic} extension of $K$, and many properties of cyclotomic extensions of $\bQ$ carry over (at least formally) to this setting. See Chapter 12 of \cite{sal07} for details of this construction and for properties of these extensions. The following proposition, a formula for the genus of cyclotomic extensions of $\bF_q(t)$, is taken from Theorem 12.7.2.

\begin{proposition}\label{cyclotomicgenus}
Let $M \in \bF_q[t]$ be monic of the form $M = \prod_{i = 1}^r P_i^{\alpha_i}$, where the $P_i$ are distinct irreducible polynomials. Then
\[
2g_M - 2 = -2\Phi(M) + \sum_{i = 1}^r d_is_i\frac{\Phi(M)}{\Phi(P_i^{\alpha_i})} + (q-2)\frac{\Phi(M)}{q-1},
\]
where $g_M$ is the genus of $K(\Lambda_M)/K$, $d_i = \deg P_i$, and $s_i = \alpha_i\Phi(P_i^{\alpha_i}) - q^{d_i(\alpha_i - 1)}$.
\end{proposition}

Finally, if a function field $L/k$ with constant field $k$ is the compositum of two subfields $K_1/k$ and $K_2/k$, we can estimate the genus of $L$ given the genera of $K_1$ and $K_2$ using Castelnuovo's inequality (Theorem 3.11.3 in \cite{sti09}), stated below.

\begin{proposition}\label{castelnuovo}
Let $K_1/k$ and $K_2/k$ be subfields of $L/k$ satisfying
\begin{itemize}
	\item $L = K_1K_2$ is the compositum of $K_1$ and $K_2$, and

	\item $[L : K_i] = n_i$ and $K_i/K$ has genus $g_i$, $i = 1, 2$.
\end{itemize}
Then the genus $g_L$ of $L/K$ is bounded by
\[
g_L \leq n_1g_1 + n_2g_2 + (n_1 - 1)(n_2 - 1).
\]
\end{proposition}

\section{Maynard-Tao over $\bF_q(t)$}

We now briefly recall the Maynard-Tao method as adapted for the function field setting in \cite{chlopt}. Fix an integer $k \geq 2$, and let $\cH = \{h_1, \ldots, h_k\}$ be an admissible $k$-tuple of elements of $\bF_q[t]$ (that is, for each prime $p \in \bF_q[t]$, the set $\{h_i \pmod p : 1 \leq i \leq k\}$ is not a complete set of residues modulo $p$). Let $W = \prod_{|p| < \log\log\log(q^\ell)} p$. Define sums $S_1$ and $S_2$ as follows:
\[
S_1 = \sum_{\substack{n \in A(q^\ell) \\ n \equiv \beta \pmod W}} \omega(n)
\]
and
\[
S_2 = \sum_{\substack{n \in A(q^\ell) \\ n \equiv \beta \pmod W}} \left(\sum_{i = 1}^k \chi_{\bP}(n + h_i) \right) \omega(n),
\]
where $A(q^\ell)$ is the set of all monic polynomials in $\bF_q[t]$ of norm $q^\ell$ (i.e., degree $\ell$), $\bP$ is the set of monic irreducible elements of $\bF_q[t]$, $\beta \in \bF_q[t]$ is chosen so that $(\beta + h_i, W) = 1$ for all $1 \leq i \leq k$ (such a $\beta$ exists by the admissibility of $\cH$), and
\[
\omega(n) = \bigg( \sum_{\substack{d_1, \ldots, d_k \\ d_i \mid (n + h_i) \forall i}} \lambda_{d_1 \ldots, d_k} \bigg)^2
\]
for suitably chosen weights $\lambda_{d_1 \ldots, d_k}$. Suppose $S_2 > (m-1)S_1$, for some integer $m \geq 2$ and some choice of weights; then there exists $n_0 \in A(q^\ell)$ such that at least $m$ of the $n_0 + h_1, \ldots, n_0 + h_k$ are prime. The goal is to find a sequence of such $n_0 \in A(q^\ell)$ as $\ell \to \infty$. If this can be done, then infinitely often we obtain gaps between primes of size at most $\max_{1 \leq i, j \leq k : i \neq j} |h_i - h_j|$. 

For the choice of suitable weights and the subsequent asymptotic formulas for $S_1$ and $S_2$, we refer to Proposition 2.3 of \cite{chlopt}, which we restate here for convenience:

\begin{proposition}\label{asymps}
Let $0 < \theta < \frac{1}{4}$ be a real number and set $R = |A(q^\ell)|^{\theta}$. Let $F$ be a piecewise differentiable real-valued function supported on the simplex $\cR_k := \{(x_1, \ldots, x_k) \in [0, 1]^k : \sum_{i = 1}^k x_i \leq 1\}$, and let 
\[ F_{\max} := \sup_{(t_1, \ldots, t_k) \in [0, 1]^k} |F(t_1, \ldots, t_k)| + \sum_{i = 1}^k |\frac{\partial F}{\partial x_i} (t_1, \ldots, t_k)|.\]
Set
\[
\lambda_{d_1, \ldots, d_k} := \left(\prod_{i = 1}^k \mu(d_i)|d_i|\right) \sum_{\substack{r_1, \ldots, r_k \\ d_i \mid r_i \forall i \\ (r_i, W) = 1 \, \forall i}} \frac{\mu(r_1, \ldots, r_k)^2}{\prod_{i=1}^k \Phi(r_i)} F\left(\frac{\log|r_1|}{\log R}, \ldots, \frac{\log|r_k|}{\log R}\right)
\]
whenever $|d_1 \cdots d_k| < R$ and $(d_1 \cdots d_k, W) = 1$, and $\lambda_{d_1, \ldots, d_k} = 0$ otherwise. Then the following asymptotic formulas hold:
\[
S_1 = \frac{(1 + o(1)) \Phi(W)^k |A(q^\ell)| (\frac{1}{\log q} \log R)^{k}}{|W|^{k+1}} I_k(F)
\]
and
\[
S_2 = \frac{(1 + o(1)) \Phi(W)^k |A(q^\ell)| (\frac{1}{\log q} \log R)^{k+1}}{\big(\log|A(q^{\ell})|\big) |W|^{k+1}} \sum_{m = 1}^k J_k^{(m)}(F),
\]
where
\[
I_k(F) := \int \cdots \int_{\cR_k} F(x_1, \ldots, x_k)^2 dx_1 \ldots dx_k,
\]
and
\[
J_k^{(m)}(F) := \int \cdots \int_{[0,1]^{k-1}} \left(\int_0^1 F(x_1, \ldots, x_k) dx_m\right)^2 dx_1 \ldots dx_{m-1}dx_{m+1} \ldots dx_k.
\]
\end{proposition}

By the above proposition, as $\ell \to \infty$, $S_2/S_1 \to \theta \frac{\sum_{m = 1}^k J_k^{(m)}(F)}{I_k(F)}$. Set
\[
M_k := \sup_F \frac{\sum_{m = 1}^k J_k^{(m)}(F)}{I_k(F)},
\]
where the supremum is taken over all $F$ satisfying the conditions of the Proposition \ref{asymps}. Following Proposition 4.13 of \cite{may15}, we have $M_k > \log k - 2\log\log k - 2$ for all large enough $k$. In particular, $M_k \to \infty$, so upon choosing $k$ large enough depending on $m$ (and choosing $F$ and $\theta$ appropriately), we obtain the desired result for any admissible $k$-tuple $\cH$.

For the present article, we fix $g$ satisfying the conditions of Theorem \ref{main} and modify the above argument as necessary; our modifications are somewhat similar to those in \cite{MR3272281}. Given an admissible $k$-tuple $\cH = \{h_1, \ldots, h_k\}$, the set $g\cH = \{gh_1, \ldots, gh_k\}$ is again admissible. We work from now on with admissible $k$-tuples $\cH$ such that every element of $\cH$ is divisible by $g$. Set
\[
W := \lcm\left(g, \prod_{|p| < \log_3(q^\ell)} p \right).
\]
With $A(q^\ell)$ defined as above, we will insist that $\ell$ is prime; this will be advantageous in what follows. We again search among $n \in A(q^{\ell})$ belonging to a certain residue class modulo $W$, but we must choose this residue class more carefully than in the original Maynard-Tao argument; that is, we choose this residue class so that primes detected by the sieve will have $g$ as a primitive root.

\begin{lemma}\label{residueclass}
We can choose $\alpha \in \bF_q[t]$ such that, for any $1 \leq i \leq k$ and for any $n \equiv \alpha \pmod W$ with $\deg(n)$ odd,
\begin{itemize}
	\item $n + h_i$ is coprime to $W$, and
	\item $\left(\frac{g}{n + h_i}\right)_{q - 1}$ generates $\bF_q^{\ast}$.
\end{itemize}
\end{lemma}

\begin{proof}
Fix a generator $\omega \in \bF_q^{\ast}$. Suppose $\deg(g)$ is even. Write $g = p_1^{f_1} \cdots p_r^{f_r}$ with $p_i$ irreducible for each $i$. Since $g$ is not an $v$th power for any $v \mid q-1$, the numbers $f_1, \ldots, f_r, q - 1$ have greatest common divisor equal to one. Hence, we may write
\[
1 = b_1f_1 + \ldots + b_rf_r + b_{r+1}(q-1)
\]
for some integers $b_i$ not all zero. Thus
\[
\omega = \omega^{b_1f_1 + \ldots + b_rf_r + b_{r+1}(q-1)} = \omega^{b_1f_1 + \ldots + b_rf_r}.
\]
Now, for each $1 \leq i \leq r$, $\omega^{b_i}$ is an element of $\bF_q^{\ast}$ of order dividing $q-1$. By Proposition \ref{properties}b, for each such $i$ there exists $a_i \in \bF_q[t]$ with $(a_i/p_i)_{q - 1} = \omega^{b_i}$; and by the Chinese remainder theorem, we can replace each $a_i$ in the system of congruences above by a single element $a \in \bF_q[t]$. So, by definition,
\[
\left(\frac{a}{g}\right)_{q - 1} = \prod_{i = 1}^r \left(\frac{a_i}{p_i}\right)^{f_i}_{q - 1} = \prod_{i = 1}^r \omega^{b_if_i} = \omega.
\]
(note that all polynomials here are monic). Choose $\alpha$ so that $\alpha \equiv a \pmod g$ and $(\alpha + h_i, W/g) = 1$ for all $h_i \in \cH$; such an $\alpha$ can be chosen by the admissibility of $\cH$. Then by Proposition \ref{properties}a, for all $n \equiv \alpha \pmod W$, we have
\[
\left(\frac{a}{g}\right)_{q - 1} = \left(\frac{\alpha + h_i}{g}\right)_{q - 1} = \left(\frac{n + h_i}{g}\right)_{q - 1},
\]
recalling that all $h_i \in \cH$ are divisible by $g$. According to Theorem \ref{reciprocitylaw},
\[
\left(\frac{n + h_i}{g}\right)_{q - 1} = (-1)^{\deg(n + h_i)\deg(g)}\left(\frac{g}{n + h_i}\right)_{q - 1} = \left(\frac{g}{n + h_i}\right)_{q - 1},
\]
so that $(\frac{g}{n + h_i})_{q - 1}$ generates $\bF_q^\ast$ as desired. If $\deg(g)$ is odd, so that the factor of $-1$ remains on the right-hand side of the above equation, repeat the argument with $-\omega$ in place of $\omega$.
\end{proof}

Let $\alpha \in \bF_q[t]$ be suitably chosen according to Lemma \ref{residueclass}. Define 
\[
\tilde{S}_1 := S_1
\]
and
\[
\tilde{S}_2 := \sum_{\substack{n \in A(q^\ell) \\ n \equiv \alpha \pmod W}} \left(\sum_{i = 1}^k \chi_{\bP_g}(n + h_i) \right) \omega(n).
\]
(So $\tilde{S}_2$ is just $S_2$ with $\bP$ replaced with $\bP_g$.) Our theorem follows immediately from the following proposition.

\begin{proposition}\label{sameasymps}
We have the same asymptotic formulas for $\tilde{S}_1$ and $\tilde{S}_2$ as we do for $S_1$ and $S_2$ in Proposition \ref{asymps}.
\end{proposition}

If we can establish Proposition \ref{sameasymps}, Maynard's argument to establish the existence of bounded rational prime gaps can be used to obtain Theorem \ref{main}.

\section{Proof of Proposition \ref{sameasymps}}

This proof follows essentially the same strategy as Section 3.3 of \cite{MR3272281}. Since $\tilde{S}_1 = S_1$, we need only concern ourselves with $\tilde{S}_2$. We can write $\tilde{S}_2 = \sum_{m = 1}^k \tilde{S}_2^{(m)}$, where
\[
\tilde{S}_2^{(m)} := \sum_{\substack{n \in A(q^\ell) \\ n \equiv \alpha \pmod W}} \chi_{\bP_g}(n + h_m) \omega(n).
\]
The proof of Proposition \ref{asymps} (which refers to Maynard's analysis) shows that, for any $m$,
\[
S_2^{(m)} \sim \frac{\varphi(W)^k|A(q^\ell)|(\frac{1}{\log q}\log R)^{k+1}}{|W|^{k+1}\log q^\ell}\cdot J_k^{(m)}(F).
\]
To establish Proposition \ref{sameasymps}, it would certainly suffice to prove that the difference between $S_2^{(m)}$ and $\tilde{S}_2^{(m)}$ is asymptotically negligible, i.e., that as $\ell \to \infty$ through prime values,
\begin{align}\label{difference}
S_2^{(m)} - \tilde{S}_2^{(m)} = o\left(\frac{\varphi(W)^k|A(q^\ell)|(\log q^\ell)^k}{|W|^{k+1}}\right).
\end{align}

We now focus on establishing (\ref{difference}) for each fixed $m$.

For prime $r$ dividing $q^\ell - 1$, let $\cP_{r}$ denote the set of all irreducible polynomials $p \in A(q^\ell)$ satisfying
\[
g^{\frac{q^{\ell} - 1}{r}} \equiv 1 \pmod p.
\]
We have the inequality
\[
0 \leq \chi_{\bP} - \chi_{\bP_g} \leq \sum_{r \mid q^{\ell} - 1} \chi_{\cP_r}
\]
for any argument which is not an irreducible polynomial dividing $g$, and it follows that
\begin{align}\label{doublesum}
0 \leq S_2^{(m)} - \tilde{S}_2^{(m)} \leq \sum_{r \mid q^{\ell} - 1} \sum_{\substack{n \in A(q^\ell) \\ n \equiv \alpha \pmod W}} \chi_{\cP_r}(n + h_m) \omega(n).
\end{align}
We will show that this double sum satisfies the asymptotic estimate in ($\ref{difference}$).

First note that primes $r$ dividing $q - 1$ make no contribution to the sum. Indeed, suppose $r \mid q-1$ and $p := n + h_m$ is detected by the sum. Then
\[
1 \equiv g^{\frac{q^{\ell} - 1}{r}} \equiv \left(\frac{g}{p}\right)_r = \left(\frac{g}{p}\right)_{q-1}^{\frac{q-1}{r}}.
\]
So $(g/p)_{q-1}$ does not generate $\bF_q^{\ast}$, and this contradicts the choice of the residue class $\alpha \pmod W$.

Upon expanding the weights and reversing the order of summation, the right-hand side of (\ref{doublesum}) becomes
\begin{align}\label{sum}
\sum_{\substack{r \mid q^\ell - 1 \\ r \nmid q - 1}} \sum_{\substack{d_1, \ldots, d_k \\ e_1, \ldots, e_k}} \lambda_{d_1 \ldots, d_k} \lambda_{e_1 \ldots, e_k} \sum_{\substack{n \in A(q^\ell) \\ n \equiv \alpha \pmod W \\ [d_i, e_i] \mid n + h_i \forall i}} \chi_{\cP_r}(n + h_m).
\end{align}
By definition of the $\lambda$ terms, the $\{d_i\}$ and $\{e_i\}$ that contribute to the sum are precisely those such that $W, [d_1, e_1], \ldots, [d_k, e_k]$ are pairwise coprime. Thus, the inner sum can be written as a sum over a single residue class modulo $M := W\prod_{i = 1}^k [d_i, e_i]$. We will also require that $n + h_m$ is coprime to $M$ (otherwise, it will not contribute to the inner sum), which occurs when $d_m = e_m = 1$.

With this in mind, we claim
\begin{align}\label{chevapp}
\sum_{\substack{n \in A(q^\ell) \\ n \equiv \alpha \pmod W \\ [d_i, e_i] \mid n + h_i \forall i}} \chi_{\cP_r}(n + h_m) = \frac{1}{r \Phi(M)} \frac{q^\ell}{\ell} + O(q^{\ell/2}).
\end{align}

Indeed, suppose $p := n + h_m$ is detected by $\chi_{\cP_r}$. Then $p$ belongs to a certain residue class modulo $M$, and $g$ is an $r$th power modulo $p$. Write $K = \bF_q(t)$. The former condition forces $\Frob_p$ to be a certain element of $\Gal(K(\Lambda_M)/K)$, and the latter condition is equivalent to $p$ splitting completely in the field $K(\zeta_r, \sqrt[r]{g})$, where $\zeta_r$ is a primitive $r$th root of unity. Let $L := K(\zeta_r, \Lambda_M, \sqrt[r]{g})$. If $K(\Lambda_M)/K$ and $K(\zeta_r, \sqrt[r]{g})/K$ are linearly disjoint extensions of $K$, then the above conditions on $p$ amount to placing $\Frob_p$ in a uniquely determined conjugacy class $\cC$ of size 1 in $\Gal(L/K)$. 

To see that $K(\Lambda_M)/K$ and $K(\zeta_r, \sqrt[r]{g})/K$ are linearly disjoint extensions of $K$, first note that since $\ell$ is prime, our conditions on $r$ imply that the order of $q$ modulo $r$ is equal to $\ell$. In particular, this means $r > \ell$. Then since $g$ is fixed while $\ell$ (and thus $r$) can be taken arbitrarily large, we can say that $g$ is not an $r$th power in $K$.

The extension $K(\sqrt[r]{g})/K$ is not Galois, since the roots of the minimal polynomial $t^r - g$ of $\sqrt[r]{g}$ are $\{\zeta_r^s \sqrt[r]{g}\}_{s = 1}^r$, where $\zeta_r$ is a primitive $r$th root of unity. If all of these roots are elements of $K$, then $K$ must contain all $r$th roots of unity, implying that $r \mid q - 1$, contradicting the conditions on the sum over values of $r$ above. Thus $K(\sqrt[r]{g}) \not\subset K(\Lambda_M)$, as $K(\Lambda_M)$ is an abelian extension of $K$, and hence any subfield, corresponding to a (normal) subgroup of $\Gal(K(\Lambda_M)/K)$, is Galois. By a theorem of Capelli on irreducible binomials,
\[
[K(\sqrt[r]{g}, \Lambda_M) : K] = [K(\sqrt[r]{g}, \Lambda_M) : K(\Lambda_M)][K(\Lambda_M) : K] = r\Phi(M).
\]
So we see that $K(\sqrt[r]{g})$ and $K(\Lambda_M)$ are linearly disjoint extensions of $K$.

For what follows, we need that $K(\sqrt[r]{g}, \Lambda_M)/K$ is a geometric extension of $K$ (i.e., that $\bF_q$ is the full constant field of $K(\sqrt[r]{g}, \Lambda_M)$). By Corollary 12.3.7 of \cite{sal07}, $K(\Lambda_M)/K$ is a geometric extension of $K$, so it is enough to show that the extension $K(\sqrt[r]{g}, \Lambda_M)/K(\Lambda_M)$ is also geometric. This follows from Proposition 3.6.6 of \cite{sti09}, provided we have that $t^r - g$ is irreducible in $K\bar{\bF}_{q}(\Lambda_M)$. The previous paragraph shows that $g$ is not an $r$th power in $K(\Lambda_M)$, so Capelli's theorem tells us $t^r - g$ is irreducible in $K(\Lambda_M)$. Now, $K\bar{\bF}_{q}(\Lambda_M)$ is a constant field extension of $K(\Lambda_M)$, the compositum of $K(\Lambda_M)$ and $\bF_{q^b}$, say. Thus, $K\bar{\bF}_{q}(\Lambda_M)/K$ is an abelian extension of $K$, as it is the compositum of two abelian extensions of $K$. If $t^r - g$ factors in this extension, then once again by Capelli, $K\bar{\bF}_{q}(\Lambda_M)/K$ must contain an $r$th root of $g$; but this is impossible, by the argument of the previous paragraph. This establishes the claim.
%

Let $K'$ denote the constant field extension $K(\zeta_r)$ of $K$; then according to Proposition 3.6.1 of \cite{sti09}, we have $[K'(\Lambda_M, \sqrt[r]{g}) : K'] = r\Phi(M)$, and hence
\[
[L:K] = [L : K'][K' : K] = [K'(\Lambda_M, \sqrt[r]{g}) : K'][K' : K] = r\Phi(M)\ell,
\]
using Proposition 10.2 of \cite{ros02} to determine $[K' : K] = \ord_q(r) = \ell$ (here $\ord_q(r)$ denotes the multiplicative order of $q$ modulo $r$). Thus $K(\zeta_r, \sqrt[r]{g})$ and $K(\Lambda_M)$ are linearly disjoint Galois extensions of $K$ with compositum $L$, as desired.
%

We are nearly in a position to use Theorem \ref{chebotarev} to estimate the sum in (\ref{chevapp}). If $\tau \in \cC$, the map $\tau$ fixes $K(\zeta_r, \sqrt[r]{g})/K$, and in particular restricts to the identity map on $\bF_{q^{\ell}}$, the constant field of $K(\zeta_r, \sqrt[r]{g})$. Now for any $a \in \bF_{q^{\ell}}$, we have
\[
\Frob_q^\ell(a) = a^{q^\ell} = a(a^{q^{\ell} - 1}) = a,
\]
and so the restriction condition of Theorem \ref{chebotarev} is satisfied. The sum in question is therefore equal to
\begin{align}\label{chebfirstapp}
\frac{1}{r\Phi(M)}\frac{q^\ell}{\ell} + O\Big(\frac{1}{r\Phi(M)}\frac{q^{\ell/2}}{\ell}\big(r\Phi(M) + g_L\big)\Big).
\end{align}

Let $g_1$ and $g_2$ denote the genus of $K'(\sqrt[r]{g})/K'$ and $K'(\Lambda_M)/K'$, respectively. 
By Proposition \ref{castelnuovo},
\[
g_L \leq \Phi(M)g_1 + rg_2 + (\Phi(M) - 1)(r-1).
\]
Recalling that $K(\sqrt[r]{g})/K$ is a geometric extension, it follows from Proposition \ref{kummergenus} that $g_1 \ll r$, with the implied constant depending on $g$. For $g_2$, we refer to Proposition \ref{cyclotomicgenus}, which states that
\[
2g_2 - 2 = -2\Phi(M) + \sum_{i = 1}^v d_is_i\frac{\Phi(M)}{\Phi(P_i^{\alpha_i})} + (q-2)\frac{\Phi(M)}{q-1},
\]
where $M = \prod_{i = 1}^v P_i^{\alpha_i}$ (with the $P_i$ distinct irreducible polynomials), $d_i = \deg P_i$, and $s_i = \alpha_i\Phi(P_i^{\alpha_i}) - q^{d_i(\alpha_i - 1)}$. At any rate, the middle sum is
\begin{align*}
\leq \Phi(M) \sum_{i = 1}^v d_i\alpha_i = \Phi(M) \sum_{i = 1}^v \alpha_i\deg(P_i) = \Phi(M)\deg(M).
\end{align*}
The first and third terms are clearly $O(\Phi(M))$, and thus $g_L \ll r\Phi(M)\log|M|$. Inserting this estimate into (\ref{chebfirstapp}), we obtain that the number of primes $p$ detected by the sum in (\ref{chevapp}) is
\begin{align}\label{halfchevapp}
\frac{1}{r\Phi(M)}\frac{q^\ell}{\ell} + O\Big(\frac{1}{r\Phi(M)}\frac{q^{\ell/2}}{\ell}\big(r\Phi(M) + r\Phi(M)\log|M|\big)\Big).
\end{align}
Recall that $M = W\prod_{i = 1}^k [d_i, e_i]$. Owing to the support of the weights $\lambda$, we have $|\prod [d_i, e_i]| < R^2$, and hence
\begin{align*}
\log|M| &= \log\Big(|W| \prod_{i = 1}^k |[d_i, e_i]|\Big) = \log|W| + \log(R^2) \\
	&\ll \log|W| + \log(q^{2\theta\ell}) \ll \ell,
\end{align*}
recalling that $W = \prod_{|p| < \log\log\log(q^\ell)} p$. Therefore the error term in (\ref{halfchevapp}) is $O(q^{\ell/2})$, as claimed.
 
Inserting the above into (\ref{sum}), we produce an $O$-term of size 
\begin{align*}
\ll q^{\ell/2}\bigg(\sum_{r | q^{\ell} - 1} 1\bigg)&\bigg(\sum_{\substack{d_1, \ldots, d_k \\ e_1, \ldots, e_k}} |\lambda_{d_1 \ldots, d_k}| |\lambda_{e_1 \ldots, e_k}|\bigg) \\ &\ll q^{\ell/2}\log(q^\ell - 1) \lambda_{\max}^2 \bigg(\sum_{s\, :\, |s| < R} \tau_k(s)\bigg)^2 \\
 &\ll q^{\ell/2} \cdot \ell \cdot R^2(\log R)^{2k},
\end{align*}
and this is $o(q^\ell)$ since $R = q^{\theta\ell}$ where $0 < \theta < 1/4$.

We now focus on the main term:
\begin{align}\label{mainterm}
\bigg(\sum_{\substack{r \mid q^{\ell} - 1 \\ r \nmid q-1}} \frac{1}{r}\bigg) \frac{q^\ell}{\ell\Phi(W)} \sideset{}{'}\sum_{\substack{d_1, \ldots, d_k \\ e_1, \ldots, e_k}} \frac{\lambda_{d_1 \ldots, d_k} \lambda_{e_1 \ldots, e_k}}{\prod_{i = 1}^k \Phi([d_i, e_i])},
\end{align}
where the $'$ on the sum means that $[d_1, e_1], \ldots, [d_k, e_k],$ and $W$ are all pairwise coprime. Recalling the support of the weights $\lambda$, this is equivalent to requiring that $(d_i, e_j) = 1$ for all $1 \leq i, j \leq k$ with $i \neq j$. We account for this by inserting the quantity $\sum_{s_{i, j} \mid d_i, e_j} \mu(s_{i, j})$, which is 1 precisely when $(d_i, e_j) = 1$ and is 0 otherwise. Define a completely multiplicative function $g$ such that $g(p) = |p| - 2$ on prime polynomials $p$; note that
\[
\frac{1}{\Phi([d_i, e_i])} = \frac{1}{\Phi(d_i)\Phi(e_i)}\sum_{u_i \mid d_i, e_i} g(u_i).
\]
Therefore, the primed sum above is equal to
\begin{align}\label{dpsum}
\sum_{u_1, \ldots, u_k}\Big(\prod_{i = 1}^k g(u_i)\Big) \sideset{}{''}\sum_{s_{1, 2}, \ldots, s_{k-1, k}} \Big(\prod_{\substack{1 \leq i, j \leq k \\ i \neq j}} \mu(s_{i, j})\Big) \sum_{\substack{d_1, \ldots, d_k \\ e_1, \ldots e_k \\ u_i \mid d_i, e_i \forall i \\ s_{i, j} \mid d_i, e_j \forall i \neq j \\ d_m = e_m = 1}} \frac{\lambda_{d_1 \ldots, d_k} \lambda_{e_1 \ldots, e_k}}{\prod_{i = 1}^k \Phi(d_i)\Phi(e_i)},
\end{align}
where the double-prime indicates that the sum is restricted to those $s_{i, j}$ which contribute to the sum, i.e. those coprime to $u_i, u_j, s_{i, a},$ and $s_{b, j}$ for all $a \neq j$ and $b \neq i$.

Define new variables
\[
y^{(m)}_{r_1, \ldots, r_k} := \left(\prod_{i = 1}^k \mu(r_i)g(r_i)\right) \sum_{\substack{d_1, \ldots, d_k \\ r_i \mid d_i \forall i \\ d_m = 1}} \frac{\lambda_{d_1, \ldots, d_k}}{\prod_{i = 1}^k \Phi(d_i)}.
\]
Then we can rewrite (\ref{dpsum}) as
\begin{align*}
\sum_{u_1, \ldots, u_k}\Big(\prod_{i = 1}^k g(u_i)\Big) \sideset{}{''}\sum_{s_{1, 2}, \ldots, s_{k-1, k}} &\Big(\prod_{\substack{1 \leq i, j \leq k \\ i \neq j}} \mu(s_{i, j})\Big) \times \\ & \Big(\prod_{i=1}^k \frac{\mu(a_i)}{g(a_i)}\Big)\Big(\prod_{j=1}^k \frac{\mu(b_j)}{g(b_j)}\Big) y^{(m)}_{a_1, \ldots, a_k} y^{(m)}_{b_1, \ldots, b_k},
\end{align*}
where $a_i = u_i\prod_{j \neq i}s_{i, j}$ and $b_j = u_j\prod_{i \neq j}s_{i, j}$. Recombining terms, we see that this is equal to
\begin{align}
\sum_{u_1, \ldots, u_k}\Big(\prod_{i = 1}^k \frac{\mu(u_i)^2}{g(u_i)}\Big) \sideset{}{''}\sum_{s_{1, 2}, \ldots, s_{k-1, k}} \Big(\prod_{\substack{1 \leq i, j \leq k \\ i \neq j}} \frac{\mu(s_{i, j})}{g(s_{i, j})^2}\Big) y^{(m)}_{a_1, \ldots, a_k} y^{(m)}_{b_1, \ldots, b_k}.
\end{align}
Let $y_{\max}^{(m)} := \max_{r_1, \ldots, r_k} |y_{r_1, \ldots, r_k}^{(m)}|$ and note that $y_{\max}^{(m)} \ll \frac{\Phi(W)}{W}\log R$; this follows from Lemma 2.6 of \cite{chlopt}. Using again the fact that $r \geq \ell$, we have
\[
\sum_{\substack{r \mid q^{\ell} - 1 \\ r \nmid q-1}} \frac{1}{r} \leq \frac{1}{\ell} \#\{\text{primes } p : p \mid q^\ell - 1\} = o(1),
\]
using the standard result that the number of distinct prime divisors of a natural number $n$ is $\ll \frac{\log n}{\log\log n}$. 

Putting everything together, we see that (\ref{mainterm}) is 
\begin{align*}
\ll \Bigg(\sum_{\substack{r \mid q^{\ell} - 1 \\ r \nmid q-1}} \frac{1}{r}\Bigg) \frac{q^\ell}{\ell\Phi(W)} &\Bigg(\sum_{\substack{u < R \\ (u, W) = 1}} \frac{\mu(u)^2}{g(u)}\Bigg)^{k-1}\Bigg(\sum_{s}\frac{\mu(s)^2}{g(s)^2}\Bigg)^{k(k-1)}\big(y_{\max}^{(m)}\big)^2 \\
&\ll \Bigg(\sum_{\substack{r \mid q^{\ell} - 1 \\ r \nmid q-1}} \frac{1}{r}\Bigg) \frac{q^\ell}{\ell\Phi(W)}\Big(\frac{\Phi(W)}{|W|}\Big)^{k+1}(\log R)^{k+1} \\
&= o\Big(q^\ell \frac{\Phi(W)^{k}}{|W|^{k+1}} (\log q^\ell)^k\Big),
\end{align*}
as desired.

\section{An example: Primitive polynomials over $\bF_2$}

We conclude by calculating an explicit bound on small gaps between primitive polynomials over $\bF_2$. Referring to the remark after Theorem 1.3 in \cite{chlopt}, any admissible 105-tuple $\cH$ of polynomials in $\bF_2[t]$ admits infinitely many shifts $f + \cH$, $f \in \bF_2[t]$, containing at least two primes. Let $\cH$ be a collection of 105 prime polynomials in $\bF_2[t]$ of norm greater than 105 (that is, of degree at least seven); it is easy to see that $\cH$ is admissible. By Gauss's formula for the number of irreducible polynomials of a given degree over a finite field, there are 104 irreducible polynomials of degree seven, eight or nine over $\bF_2$, so take $\cH$ to be a 105-tuple of primes of degree at least seven and at most ten.

To apply our method, we require in general that each element of $\cH$ be a multiple of the given primitive root $g$, and we may modify an admissible tuple $\cH$ to obtain an appropriate admissible tuple by replacing each $h \in \cH$ by $gh$. In the present case, with $g = t$ and $\cH$ the 105-tuple described above, this operation results in an admissible 105-tuple $\cH$ of polynomials of degree at least eight and at most eleven. Thus, with this choice of $\cH = \{h_1, h_2, \ldots, h_{105}\}$, one finds that there are infinitely many gaps of norm at most $N$ between primitive polynomials, where
\[
N \leq \max_{1 \leq i \neq j \leq 105} |h_i - h_j| \leq 2^{11} = 2048.
\]
For other choices of $g$ and $q$, this construction produces a bound of the form $q^{\deg(g) + 10}$.
\section*{Acknowledgements}

The author thanks Paul Pollack for guidance during the course of this project and for many helpful comments during the editing of this manuscript.

\bibliographystyle{amsalpha}
\bibliography{refs}

\end{document}